\documentclass[11pt,letterpaper]{amsart}
\usepackage[left=2cm,tmargin=2cm,bmargin=2cm,right=2cm]{geometry}

\usepackage{graphicx} 
\usepackage{tikz}
\usepackage{amssymb}
\usepackage{amsmath}
\usepackage{amsthm}
\usepackage{tikz-cd}
\usepackage{array}
\usepackage{mathtools}
\usepackage{adjustbox}
\usepackage{float}
\usetikzlibrary {positioning}
\usepackage{biblatex}
\newcolumntype{P}[1]{>{\centering\arraybackslash}p{#1}}

\newtheorem{theorem}{Theorem}

\newcommand{\R}{\mathbb{R}}

\begin{document}

\title{A note on right-angled, ideal hyperbolic polyhedra}

\author{Ian Whitehead}
\email{whiteh2@stolaf.edu}
\address{Department of Mathematics, Statistics, and Computer Science, St. Olaf College, Northfield, MN}

\begin{abstract}
    We prove a necessary and sufficient condition for an abstract polyhedron $\Pi$ to be realized as a right-angled, ideal polyhedron in hyperbolic space.
\end{abstract}
\maketitle
\section{Introduction}

Hyperbolic polyhedra which are \emph{right-angled} (all dihedral angles $=\tfrac{\pi}{2}$) and \emph{ideal} (all vertices at $\infty$) play a special role in three-dimensional hyperbolic geometry. In \cite{Andreev} (see also \cite{RHD}), Andreev gives combinatorial conditions for an abstract polyhedron $\Pi$ to be realizable in $\mathbb{H}^3$ with given dihedral angles. Specialized to the case of right-angled, ideal polyhedra, Andreev's theorem is as follows:
\begin{theorem}[\cite{Andreev}, Theorem 1.2] \label{Theorem1}
An abstract polyhedron $\Pi$ can be realized as a right-angled, ideal hyperbolic polyhedron if and only if it satisfies the following conditions:
\begin{enumerate}
\item[\textbf{m1}] All vertices of $\Pi$ have degree 4. 
\item[\textbf{m3}] $\Pi$ has no ``rectangular prismatic elements,'' i.e. quadruples of faces $F_1$, $F_2$, $F_3$, $F_4$ such that the pairs $(F_1, F_2)$, $(F_2, F_3)$, $(F_3, F_4)$, and $(F_4, F_1)$ each share a common edge, but the pairs $(F_1,F_3)$, $(F_2, F_4)$ do not share a common edge or vertex. 
\item[\textbf{m5}]  $\Pi$ has no triples of faces $F_1$, $F_2$, $F_3$ such that the pairs $(F_1, F_2)$ and $(F_2, F_3)$ each share a common edge, and $(F_1, F_3)$ share a common vertex which is not on $F_2$. 
\end{enumerate}
\end{theorem}
\noindent Andreev also lists conditions \textbf{m0}, \textbf{m2}, \textbf{m4}, which are vacuous for right-angled, ideal polyhedra. 

We prove a different equivalent condition for $\Pi$ to be realizable as a right-angled, ideal hyperbolic polyhedron. The \emph{rectification} of a polyhedron $\Pi_0$ is another polyhedron $\Pi$ with a vertex for each edge of $\Pi_0$ and a face for each vertex and face of $\Pi_0$. It is constructed by truncating every vertex of $\Pi_0$ with a cut through the midpoints of the neighboring edges. For example, the rectification of a tetrahedron is an octahedron, and the rectification of a cube or octahedron is a cuboctahedron. See Figure \ref{Figure1}.

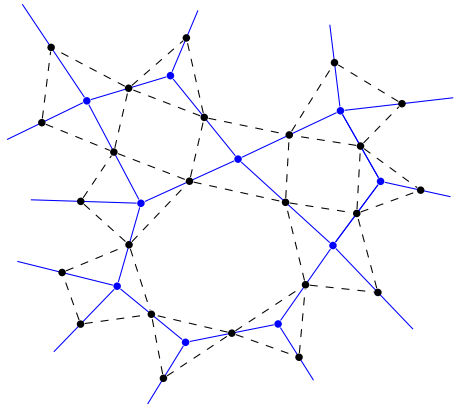
\begin{figure}[h]
\vspace{-.5cm}
\begin{tikzpicture}
\node[circle,fill=blue,inner sep=1pt] (B1) at (2.1179,2.58502) {};

\node[circle,fill=blue,inner sep=1pt] (B2) at (3.02296,1.84241) {};

\node[circle,fill=blue,inner sep=1pt] (B3) at (2.43233,3.682) {};

\node[circle,fill=blue,inner sep=1pt] (B4) at (4.24698,2.08529) {};

\node[circle,fill=blue,inner sep=1pt] (B5) at (4.97298,3.12332) {};

\node[circle,fill=blue,inner sep=1pt] (B6) at (3.71747,4.26612) {};

\node[circle,fill=blue,inner sep=1pt] (B7) at (1.71656,5.03713) {};

\node[circle,fill=blue,inner sep=1pt] (B8) at (2.82288,5.37015) {};

\node[circle,fill=blue,inner sep=1pt] (B9) at (5.60253,3.97203) {};

\node (B10) at (3.24419,6.36852) {};

\node (B11) at (4.91456,6.17912) {};

\node[circle,fill=blue,inner sep=1pt] (B12) at (5.07174,4.90536) {};

\node (B13) at (1.14977,1.5817) {};

\node (B14) at (0.6616,2.94717) {};

\node (B15) at (2.43777,0.890967) {};

\node (B16) at (0.840896,3.73223) {};

\node (B17) at (4.80093,1.20551) {};

\node (B18) at (6.16137,1.88154) {};

\node (B19) at (0.526339,4.45976) {};

\node (B20) at (0.77259,6.4505) {};

\node (B25) at (6.69942,5.09354) {};

\node (B26) at (6.66384,3.73875) {};

\draw[blue] (B1) -- (B2);

\draw[blue] (B2) -- (B4);

\draw[blue] (B1) -- (B3);

\draw[blue] (B4) -- (B5);

\draw[blue] (B5) -- (B6);

\draw[blue] (B3) -- (B6);

\draw[blue] (B5) -- (B9);

\draw[blue] (B9) -- (B12);

\draw[blue] (B6) -- (B12);

\draw[blue] (B3) -- (B7);

\draw[blue] (B7) -- (B8);

\draw[blue] (B6) -- (B8);

\draw[blue] (B5) -- (B9);

\draw[blue] (B9) -- (B12);

\draw[blue] (B8) -- (B10);

\draw[blue] (B11) -- (B12);

\draw[blue] (B1) -- (B13);

\draw[blue] (B1) -- (B14);

\draw[blue] (B2) -- (B15);

\draw[blue] (B3) -- (B16);

\draw[blue] (B4) -- (B17);

\draw[blue] (B5) -- (B18);

\draw[blue] (B7) -- (B19);

\draw[blue] (B7) -- (B20);

\draw[blue] (B12) -- (B25);

\draw[blue] (B9) -- (B26);

\node[circle,fill=black,inner sep=1pt] (C1) at (1.63384,2.08336) {};

\node[circle,fill=black,inner sep=1pt] (C2) at (2.73037,1.36669) {};

\node[circle,fill=black,inner sep=1pt] (C3) at (1.38975,2.7661) {};

\node[circle,fill=black,inner sep=1pt] (C4) at (2.57043,2.21372) {};

\node[circle,fill=black,inner sep=1pt] (C5) at (1.63661,3.70711) {};

\node[circle,fill=black,inner sep=1pt] (C6) at (2.27511,3.13351) {};

\node[circle,fill=black,inner sep=1pt] (C7) at (3.63497,1.96385) {};

\node[circle,fill=black,inner sep=1pt] (C8) at (2.07444,4.35956) {};

\node[circle,fill=black,inner sep=1pt] (C9) at (4.52395,1.6454) {};

\node[circle,fill=black,inner sep=1pt] (C10) at (1.12145,4.74845) {};

\node[circle,fill=black,inner sep=1pt] (C11) at (3.0749,3.97406) {};

\node[circle,fill=black,inner sep=1pt] (C12) at (4.60998,2.6043) {};

\node[circle,fill=black,inner sep=1pt] (C13) at (5.56717,2.50243) {};

\node[circle,fill=black,inner sep=1pt] (C14) at (4.34522,3.69472) {};

\node[circle,fill=black,inner sep=1pt] (C15) at (3.27017,4.81813) {};

\node[circle,fill=black,inner sep=1pt] (C16) at (2.26972,5.20364) {};

\node[circle,fill=black,inner sep=1pt] (C17) at (5.28776,3.54768) {};

\node[circle,fill=black,inner sep=1pt] (C18) at (1.24457,5.74382) {};

\node[circle,fill=black,inner sep=1pt] (C19) at (5.33713,4.43869) {};

\node[circle,fill=black,inner sep=1pt] (C20) at (3.03353,5.86933) {};

\node[circle,fill=black,inner sep=1pt] (C22) at (4.3946,4.58574) {};

\node[circle,fill=black,inner sep=1pt] (C23) at (6.13319,3.85539) {};

\node[circle,fill=black,inner sep=1pt] (C25) at (5.88558,4.99945) {};

\node[circle,fill=black,inner sep=1pt] (C27) at (4.99315,5.54224) {};

\draw[dashed] (C1) -- (C3);

\draw[dashed] (C1) -- (C4);

\draw[dashed] (C2) -- (C4);

\draw[dashed] (C2) -- (C7);

\draw[dashed] (C3) -- (C6);

\draw[dashed] (C5) -- (C6);

\draw[dashed] (C5) -- (C8);

\draw[dashed] (C6) -- (C4);

\draw[dashed] (C6) -- (C11);

\draw[dashed] (C7) -- (C4);

\draw[dashed] (C7) -- (C9);

\draw[dashed] (C7) -- (C12);

\draw[dashed] (C8) -- (C11);

\draw[dashed] (C8) -- (C16);

\draw[dashed] (C9) -- (C12);

\draw[dashed] (C10) -- (C8);

\draw[dashed] (C10) -- (C18);

\draw[dashed] (C11) -- (C14);

\draw[dashed] (C11) -- (C15);

\draw[dashed] (C12) -- (C13);

\draw[dashed] (C12) -- (C14);

\draw[dashed] (C13) -- (C17);

\draw[dashed] (C14) -- (C17);

\draw[dashed] (C14) -- (C22);

\draw[dashed] (C15) -- (C20);

\draw[dashed] (C15) -- (C22);

\draw[dashed] (C15) -- (C16);

\draw[dashed] (C16) -- (C20);

\draw[dashed] (C17) -- (C23);

\draw[dashed] (C16) -- (C18);

\draw[dashed] (C17) -- (C19);

\draw[dashed] (C19) -- (C22);

\draw[dashed] (C19) -- (C23);

\draw[dashed] (C19) -- (C25);

\draw[dashed] (C22) -- (C27);

\draw[dashed] (C25) -- (C27);

\end{tikzpicture}
\vspace{-.5cm}
\caption{Planar graph $\Pi_0$ (blue) and rectification $\Pi$ (black, dashed).}
\label{Figure1}
\end{figure}

Our main result is the following.
\begin{theorem} \label{Theorem2}
An abstract polyhedron $\Pi$ can be realized as a right-angled, ideal hyperbolic polyhedron if and only if it is the rectification of some other polyhedron $\Pi_0$. 
\end{theorem}
This is more or less a corollary of Theorem \ref{Theorem1}, but is of interest because it gives a more geometric characterization of $\Pi$.  We also give a short proof of the celebrated Koebe-Andreev-Thurston Theorem, which is closely related to the ``if'' part of Theorem \ref{Theorem2}. Thurston derives this result from Andreev's Theorem \cite[Corollary 13.6.2]{Thurston}. Our result clarifies the relationship between these theorems. 

\section{Proof of Theorem \ref{Theorem2}}

Any polyhedral graph $\Pi_0$ admits a rectification $\Pi$ which is 4-valent. Conversely, given a 4-valent polyhedral graph $\Pi$, its faces can be 2-colored, say red and blue. We can form a new planar graph $\Pi_0$ with a vertex for each blue face, and an edge wherever two blue faces share a vertex. Then $\Pi$ is the rectification of $\Pi_0$. By Steinitz's theorem, $\Pi_0$ is the graph of a polyhedron if and only if it is 3-connected. We will show that $\Pi$ satisfies conditions \textbf{m3} and \textbf{m5} if and only if $\Pi_0$ is 3-connected. 

First, suppose that $\Pi$ is a polyhedral graph satisfying conditions \textbf{m1}, \textbf{m3}, \textbf{m5}. In order to establish that the blue face graph $\Pi_0$ is 3-connected, suppose that two blue faces $B_1$, $B_2$ are deleted. Since the dual of $\Pi$ is 3-connected, for any two blue faces $B$, $B'$ not equal to $B_1$, $B_2$, we can find a path from $B$ to $B'$ which avoids $B_1$ and $B_2$. The path consists of faces $B=F_1, F_2, \ldots F_r=B'$ such that each $F_i$ is adjacent to $F_{i+1}$ and no $F_i=B_1$ or $B_2$. Because the faces are two-colored, this path must alternate between blue and red faces. 

We now transform this path to a path in $\Pi_0$ which still avoids $B_1$ and $B_2$. If $F_i$ is a red face in the path, then it is surrounded by adjacent blue faces, including $F_{i-1}$ and $F_{i+1}$ We can form a path from $F_{i-1}$ to $F_{i+1}$ in $\Pi_0$ by following the chain of consecutive adjacent blue faces between them. See Figure \ref{Figure2}.

\begin{figure}[h]
\begin{tikzpicture}
\coordinate (C1) at (1.63384,2.08336) {};

\coordinate (C2) at (2.73037,1.36669) {};

\coordinate (C3) at (1.38975,2.7661) {};

\coordinate (C4) at (2.57043,2.21372) {};

\coordinate (C5) at (1.63661,3.70711) {};

\coordinate (C6) at (2.27511,3.13351) {};

\coordinate (C7) at (3.63497,1.96385) {};

\coordinate (C8) at (2.07444,4.35956) {};

\coordinate (C9) at (4.52395,1.6454) {};

\coordinate (C10) at (1.12145,4.74845) {};

\coordinate (C11) at (3.0749,3.97406) {};

\coordinate (C12) at (4.60998,2.6043) {};

\coordinate (C13) at (5.56717,2.50243) {};

\coordinate (C14) at (4.34522,3.69472) {};

\coordinate (C15) at (3.27017,4.81813) {};

\coordinate (C16) at (2.26972,5.20364) {};

\coordinate (C17) at (5.28776,3.54768) {};

\coordinate (C18) at (1.24457,5.74382) {};

\coordinate (C19) at (5.33713,4.43869) {};

\coordinate (C20) at (3.03353,5.86933) {};

\coordinate (C21) at (2.31589,6.69906) {};

\coordinate (C22) at (4.3946,4.58574) {};

\coordinate (C23) at (6.13319,3.85539) {};

\coordinate (C24) at (4.07938,6.27382) {};

\coordinate (C25) at (5.88558,4.99945) {};

\coordinate (C26) at (3.22204,7.00798) {};

\coordinate (C27) at (4.99315,5.54224) {};

\coordinate (C28) at (5.62024,6.2949) {};

\coordinate (C29) at (4.90994,6.87269) {};

\draw[fill=blue!10] (C1) -- (C3) -- (C6) -- (C4) -- cycle;

\draw[fill=blue!10] (C2) -- (C4) -- (C7) -- cycle;

\draw[fill=blue!10] (C5) -- (C6) -- (C11) -- (C8) -- cycle;

\draw[fill=blue!10] (C7) -- (C9) -- (C12) -- cycle;

\draw[fill=blue!10] (C12) -- (C13) -- (C17) -- (C14) -- cycle;

\draw[fill=blue!10] (C11) -- (C14) -- (C22) -- (C15) -- cycle;

\draw[fill=red!10] (C4) -- (C6) -- (C11) -- (C14) -- (C12) -- (C7) -- cycle;

\draw[fill=red!10] (C8) -- (C11) -- (C15) -- (C16) -- cycle;

\draw[fill=red!10] (C9) -- (C12) -- (C13) -- cycle;

\draw[fill=red!10] (C14) -- (C17) -- (C19) -- (C22) -- cycle;

\draw[fill=red!10] (C1) -- (1.69377, 1.23634) -- (C2) -- (C4) --  cycle;

\draw[fill=red!10] (C3) -- (C5) -- (C6) -- cycle;

\draw[fill=red!10] (C2) -- (C7) -- (C9) --(4.51935, 1.04824) -- (3.01935, 1.04824) -- cycle;

\node[blue] (B1) at (2.0179,2.48502) {$F_{i-1}$};

\node (B3) at (2.43233,3.682) {};

\node (B6) at (3.71747,4.26612) {};

\node[red!75!black] (R) at (3.44607, 2.96105) {$F_i$};

\node[blue] (B5) at (4.97298,3.12332) {$F_{i+1}$};

\draw[blue, ->] (B1) -- (B3);

\draw[blue, ->] (B3) -- (B6);

\draw[blue, ->] (B6) -- (B5);
\end{tikzpicture}
\caption{Following a chain of adjacent blue faces around a red face.}
\label{Figure2}
\end{figure}
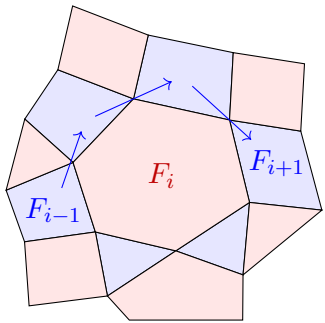

Since we can go clockwise or counterclockwise, this new path is possible unless $B_1$ and $B_2$ are both adjacent to $F_i$, one between $F_{i-1}$ and $F_{i+1}$ in the clockwise direction, and one in the counterclockwise direction. This means there is no vertex shared by $B_1$, $B_2$, and $F_i$. If this occurs, then condition \textbf{m5} implies that $B_1$, $B_2$ do not share a vertex, and condition \textbf{m3} implies that there is no red face besides $F_i$ which is adjacent to both $B_1$ and $B_2$. 

In this case, we further transform the path as follows. First we form the path from $F_{i-1}$ to $F_{i+1}$ along the chain of blue faces adjacent to $F_i$, passing through $B_1$. Then, to avoid $B_1$, we follow the chain of red faces other than $F_i$ around it. Finally, to avoid these red faces, we follow the chain of all blue faces other than $B_1$ around each one. See Figure \ref{Figure3}. Conditions \textbf{m3} and \textbf{m5} imply that this modified path does not go through $B_2$. We have shown that $\Pi_0$ is 3-connected. 

\begin{figure}[h]
\begin{tikzpicture}
\coordinate (C1) at (1.63384,2.08336) {};

\coordinate (C2) at (2.73037,1.36669) {};

\coordinate (C3) at (1.38975,2.7661) {};

\coordinate (C4) at (2.57043,2.21372) {};

\coordinate (C5) at (1.63661,3.70711) {};

\coordinate (C6) at (2.27511,3.13351) {};

\coordinate (C7) at (3.63497,1.96385) {};

\coordinate (C8) at (2.07444,4.35956) {};

\coordinate (C9) at (4.52395,1.6454) {};

\coordinate (C10) at (1.12145,4.74845) {};

\coordinate (C11) at (3.0749,3.97406) {};

\coordinate (C12) at (4.60998,2.6043) {};

\coordinate (C13) at (5.56717,2.50243) {};

\coordinate (C14) at (4.34522,3.69472) {};

\coordinate (C15) at (3.27017,4.81813) {};

\coordinate (C16) at (2.26972,5.20364) {};

\coordinate (C17) at (5.28776,3.54768) {};

\coordinate (C18) at (1.24457,5.74382) {};

\coordinate (C19) at (5.33713,4.43869) {};

\coordinate (C20) at (3.03353,5.86933) {};

\coordinate (C21) at (2.31589,6.69906) {};

\coordinate (C22) at (4.3946,4.58574) {};

\coordinate (C23) at (6.13319,3.85539) {};

\coordinate (C24) at (4.07938,6.27382) {};

\coordinate (C25) at (5.88558,4.99945) {};

\coordinate (C26) at (3.22204,7.00798) {};

\coordinate (C27) at (4.99315,5.54224) {};

\coordinate (C28) at (5.62024,6.2949) {};

\coordinate (C29) at (4.90994,6.87269) {};

\draw[fill=blue!10] (C1) -- (C3) -- (C6) -- (C4) -- cycle;

\draw[fill=blue!10] (C2) -- (C4) -- (C7) -- cycle;

\draw[fill=blue!10] (C5) -- (C6) -- (C11) -- (C8) -- cycle;

\draw (C7) -- (C9) -- (C12) -- cycle;

\draw[fill=blue!10] (C12) -- (C13) -- (C17) -- (C14) -- cycle;

\draw (C11) -- (C14) -- (C22) -- (C15) -- cycle;

\draw[fill=blue!10] (C8) -- (C10) -- (C18) -- (C16) -- cycle;

\draw[fill=blue!10] (C15) -- (C16) -- (C20) -- cycle;

\draw[fill=blue!10] (C17) -- (C19) -- (C23) -- cycle;

\draw[fill=blue!10] (C20) -- (C21) -- (C26) -- (C24) -- cycle;

\draw[fill=blue!10] (C24) -- (C27) -- (C28) -- (C29) -- cycle;

\draw[fill=blue!10] (C19) -- (C22) -- (C27) -- (C25) -- cycle;

\draw[fill=red!10] (C4) -- (C6) -- (C11) -- (C14) -- (C12) -- (C7) -- cycle;

\draw[fill=red!10] (C8) -- (C11) -- (C15) -- (C16) -- cycle;

\draw[fill=red!10] (C9) -- (C12) -- (C13) -- cycle;

\draw[fill=red!10] (C14) -- (C17) -- (C19) -- (C22) -- cycle;

\draw[fill=red!10] (C1) -- (1.69377, 1.23634) -- (C2) -- (C4) --  cycle;

\draw[fill=red!10] (C3) -- (C5) -- (C6) -- cycle;

\draw[fill=red!10] (C2) -- (C7) -- (C9) --(4.51935, 1.04824) -- (3.01935, 1.04824) -- cycle;

\draw[fill=red!10] (C15) -- (C20) -- (C24) -- (C27) -- (C22) -- cycle;

\draw[fill=red!10] (C16) -- (C18) -- (C21) -- (C20) -- cycle;

\draw[fill=red!10] (C5) -- (C8) --  (C10) -- (0.76134, 4.09793) -- cycle;

\draw[fill=red!10] (C26) -- (C24) -- (C29) --  (4.48688, 7.29651) --(3.78688, 7.29651) -- cycle;

\draw[fill=red!10] (C28) -- (C27) -- (C25)  -- (6.54918, 5.94587) -- cycle;

\draw[fill=red!10] (C25) -- (C19) -- (C23)  -- cycle;

\draw[fill=red!10] (C13) -- (C17) -- (C23) -- (6.51295, 3.72211) -- (6.11295, 2.22211) -- cycle;

\node[blue] (B1) at (2.0179,2.48502) {$F_{i-1}$};

\node (B2) at (3.02296,1.84241) {};

\node (B3) at (2.43233,3.682) {};

\node[blue] (B4) at (4.24698,2.03529) {$B_2$};

\node[blue] (B5) at (4.97298,3.12332) {$F_{i+1}$};

\node[blue] (B6) at (3.71747,4.26612) {$B_1$};

\node (B7) at (1.71656,5.03713) {};

\node (B8) at (2.82288,5.37015) {};

\node (B9) at (5.60253,3.97203) {};

\node (B10) at (3.24419,6.36852) {};

\node (B11) at (4.91456,6.17912) {};

\node (B12) at (5.07174,4.90536) {};

\node (B3) at (2.43233,3.682) {};

\node (B6) at (3.71747,4.26612) {};

\node[red!75!black] (R) at (3.44607, 2.96105) {$F_i$};

\draw[blue, ->] (B1) -- (B3);
\draw[blue, ->] (B3) -- (B7);
\draw[blue, ->] (B7) -- (B8);
\draw[blue, ->] (B8) -- (B10);
\draw[blue, ->] (B10) -- (B11);
\draw[blue, ->] (B11) -- (B12);
\draw[blue, ->] (B12) -- (B9);
\draw[blue, ->] (B9) -- (B5);
\end{tikzpicture}
\caption{Avoiding two blue faces adjacent to a red face.}
\label{Figure3}
\end{figure}
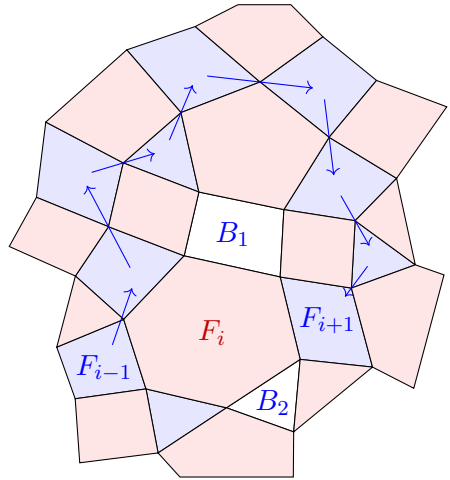

For the opposite implication, we assume $\Pi$ is a 4-valent polyhedral graph not satisfying all of Andreev's conditions, and show that $\Pi_0$ is not 3-connected. Suppose $\Pi$ does not satisfy \textbf{m5}. For the triple of faces $F_1$, $F_2$, $F_3$ violating \textbf{m5}, suppose $F_1$ and $F_3$ are blue while $F_2$ is red. Draw a simple closed curve $\gamma$ starting in $F_1$, crossing the edge into $F_2$, crossing the edge into $F_3$, and crossing the vertex back into $F_1$. Since there are faces adjacent to $F_1$, $F_2$, and $F_3$ on both sides of $\gamma$, there are both blue and red faces on both sides of $\gamma$. If we choose two blue faces on opposite sides of $\gamma$, any path between them in $\Pi_0$ must cross $\gamma$, and must therefore include $F_1$ or $F_3$. If we remove $F_1$ and $F_3$, we disconnect $\Pi_0$. Thus $\Pi_0$ is not 3-connected. If $F_1$ and $F_3$ are red while $F_2$ is blue, we find that the dual of $\Pi_0$ is not 3-connected, which also implies that $\Pi_0$ is not 3-connected. The proof when $\Pi$ does not satisfy \textbf{m3} is similar.

\section{The Koebe-Andreev-Thurston Theorem} 

There are many equivalent statements of the Koebe-Andreev-Thurston theorem, and many proofs in the literature. See \cite{Stephenson} for an overview. We will sketch a proof of the following version.

\begin{theorem}[Koebe-Andreev-Thurston] \label{Theorem3}
Any abstract polyhedron $\Pi_0$ admits a realization in $\R^3$ with a midsphere, i.e. a sphere tangent to every edge.
\end{theorem} 

\begin{proof}
By Theorem \ref{Theorem2}, we can realize the rectification $\Pi$ of $\Pi_0$ as a right-angled, ideal hyperbolic polyhedron. We can color the faces of $\Pi$ blue if they correspond to vertices of $\Pi_0$ and red if they correspond to faces of $\Pi_0$. In the Poincar\'{e} ball model of $\mathbb{H}^3$, the faces of $\Pi$ are cut out by spheres in $\R^3$ which are orthogonal to the boundary sphere $S$. We now realize $\Pi_0$ in $\R^3$ by placing each vertex at the center of the corresponding blue sphere. Each edge of $\Pi_0$ is a segment between the centers of two tangent blue spheres. The point of tangency is an ideal vertex of $\Pi$. At this point, $S$ is orthogonal to both spheres, hence tangent to the edge connecting their centers. 
\end{proof}

Thurston's statement of the theorem is in terms of circle configurations with certain tangency properties. Here, the blue and red spheres intersect $S$ in a pair of circle configurations with tangencies governed by $\Pi_0$ and its dual.

\printbibliography

@article {Andreev,
    AUTHOR = {Andreev, E. M.},
     TITLE = {Convex polyhedra of finite volume in {L}oba\v{c}evski\u{i} space},
   JOURNAL = {Mat. Sb. (N.S.)},
    VOLUME = {83 (125)},
      YEAR = {1970},
     PAGES = {256--260},
   MRCLASS = {50.40 (52.00)},
  MRNUMBER = {0273510},
MRREVIEWER = {M. Decuyper},
}

@article {RHD,
    AUTHOR = {Roeder, Roland K. W. and Hubbard, John H. and Dunbar, William
              D.},
     TITLE = {Andreev's theorem on hyperbolic polyhedra},
   JOURNAL = {Ann. Inst. Fourier (Grenoble)},
  FJOURNAL = {Universit\'e{} de Grenoble. Annales de l'Institut Fourier},
    VOLUME = {57},
      YEAR = {2007},
    NUMBER = {3},
     PAGES = {825--882},
      ISSN = {0373-0956,1777-5310},
   MRCLASS = {51M10 (51F15 52B10 57M50)},
  MRNUMBER = {2336832},
MRREVIEWER = {Joan\ Porti},
       DOI = {10.5802/aif.2279},
       URL = {https://doi.org/10.5802/aif.2279},
}

@book {Stephenson,
    AUTHOR = {Stephenson, Kenneth},
     TITLE = {Introduction to circle packing},
      NOTE = {The theory of discrete analytic functions},
 PUBLISHER = {Cambridge University Press, Cambridge},
      YEAR = {2005},
     PAGES = {xii+356},
      ISBN = {978-0-521-82356-2; 0-521-82356-0},
   MRCLASS = {52C26 (30G25)},
  MRNUMBER = {2131318},
MRREVIEWER = {Fr\'{e}d\'{e}ric Math\'{e}us},
}

@book{Thurston,
  title={The Geometry and Topology of Three-manifolds},
  author={Thurston, W.P.},
  url={https://archive.org/details/ThurstonTheGeometryAndTopologyOfThreeManifolds},
  year={1979},
  publisher={Unpublished}
}

\end{document}